\documentclass[reqno]{amsart}
\usepackage{latexsym,amsfonts,amssymb,amsmath,amsthm,amscd,enumitem}

\usepackage[usenames,dvipsnames]{xcolor}
\usepackage[colorlinks=true]{hyperref}
\usepackage{tikz-cd}
\usepackage{soul}  

\usepackage[normalem]{ulem}

\usepackage{mathtools}

\usepackage{enumitem}

\usepackage{mdwlist}

\usepackage{comment}
\usepackage[dvipsnames]{xcolor}

\usepackage{bbm}
\usepackage{hyperref}

\def\dotfill#1{\cleaders\hbox to #1{.}\hfill}
\makeatletter
\def\myrulefill{\leavevmode\leaders\hrule height .7ex width 1ex depth -0.6ex\hfill\kern\z@}
\makeatother 

\theoremstyle{plain}
\newtheorem{theorem}{Theorem}[section]

\newtheorem{lemma}{Lemma}[section]

\newtheorem{claim}{Claim}

\newtheorem{definition}{Definition}[section]

\numberwithin{equation}{section}
\numberwithin{definition}{section}

\DeclareFontFamily{U}{mathx}{\hyphenchar\font45}
\DeclareFontShape{U}{mathx}{m}{n}{
      <5> <6> <7> <8> <9> <10>
      <10.95> <12> <14.4> <17.28> <20.74> <24.88>
      mathx10
      }{}
\DeclareSymbolFont{mathx}{U}{mathx}{m}{n}
\DeclareFontSubstitution{U}{mathx}{m}{n}
\DeclareMathAccent{\widecheck}{0}{mathx}{"71}
\DeclareMathAccent{\wideparen}{0}{mathx}{"75}

\newcommand\JM{Mierczy\'nski}
\newcommand\RR{\ensuremath{\mathbb{R}}}

\newcommand\NN{\ensuremath{\mathbb{N}}}
\newcommand\PP{\ensuremath{\mathbb{P}}}

\newcommand{\OFP}{\ensuremath{(\Omega,\mathfrak{F},\PP)}}

\newcommand{\norm}[1]{\ensuremath{\lVert#1\rVert}}

\DeclareMathOperator{\lnplus}{ln^{+}}

\DeclareMathOperator{\spanned}{span}

\definecolor{afb}{rgb}{0.36, 0.54, 0.66}
\definecolor{o}{RGB}{242, 138, 2}
\definecolor{bur}{rgb}{0.5, 0.0, 0.13}
\definecolor{ph}{rgb}{0.87, 0.0, 1.0}
\definecolor{pers}{rgb}{0.0, 0.65, 0.58}

\begin{document}

\title {Oseledets Decomposition on Sub semiflows}
\author{Marek Kryspin}
\address{Faculty of Pure and Applied Mathematics,
Wroc{\l}aw University of Science and Technology,
Wybrze\.ze Wyspia\'nskiego 27, PL-50-370 Wroc{\l}aw, Poland.}
\email{marek.kryspin@pwr.edu.pl}
\thanks{The author is supported by the National Science Centre, Poland (NCN) under the grant Sonata Bis with a number NCN 2020/38/E/ST1/00153}
\begin{abstract}
The existence of the Oseledets decomposition on continuously embedded subspaces of Banach spaces is proved in this paper. Natural assumptions facilitating such transfer of the Oseledets decomposition are presented, notably conditions often met by dynamical systems generated by differential equations.
\end{abstract}
\maketitle

\section*{Introduction} 
The primary objective of this work is to delineate natural assumptions facilitating the transfer of the Oseledets decomposition from a Banach space into another Banach space continuously embedded in the former. In general, the Oseledets-type decomposition implies the partitioning/splitting (in the form of a direct sum) of the fiber space, within which a dynamical system operates, into finite-dimensional subspaces, potentially an infinite many of them, each corresponding to a specific Lyapunov exponent. This process gives rise to a hierarchy of subspaces often referred to as the Oseledets filtration or flags. Lyapunov exponents hold paramount significance in the realm of systems dynamics, as they dictate the exponential asymptotic growth rate along trajectories. For results, see, e.g., \cite{Doan-1, F-L-Q, F-S, GTQu, GTQ-2, Lee, Lian-Lu, Varz-Ried}.

There are numerous papers that concentrate on the subject of dynamical systems generated by differential equations and various phase space decompositions; for instance~\cite{Ba-Piazz, Ch-Latushkin, Chow-Leiva-1, Chow-Leiva-2, Doan-1, Doan-2, Marek-Janusz, MiShPart1, MiShPart2, MiShPart3, Martin-Smith-1, Martin-Smith-2}. In many cases of differential equations there is no ``natural'' phase space.  Nonetheless, it appears that there is a scarcity of research papers addressing the topics of regularization and the transfer of Oseledets-type decomposition.   As examples, we can mention here, first, ordinary or partial differential equations with delay, as considered in \cite{Marek-Janusz-Sylvia-Rafa, MiNoOb1, MiNoOb2}, and, second, advection diffusion equations and others, as investigated in \cite{Blu-PS}. The paper~\cite{GTQ-3} is also worth mentioning, where the authors discuss the possibility of transferring the Oseledets decomposition to dense subspaces of the fiber space with separable dual.

The presuppositions of this paper, meaning the requirements set for the solving operators and the fiber space, are often considerably weaker than the specific properties of these (for instance, compactness or separable dual are not necessarily required).  On the other hand, such relaxation of assumptions may lead to new results. To be precise, decompositions in subspaces (in particular, subspaces with a finer/stronger topology are of particular interest here) of Banach spaces (in terms of scales and continuous embeddings) frequently encountered in real-world problems related to differential equations, may arise from this approach. To the author's knowledge, there are no known theorems allowing the transfer of Oseledets decompositions to subspaces without assuming the previously mentioned separability of the fiber space and/or its dual.

To give a flavor of our results, we formulate now some specializations of our main results. It is well-known that a linear ordinary differential equation with delay generates a dynamical system that admits the Oseledets splitting in the fiber space $L_p([-1,0],\RR^N)\oplus\RR^N$. This follows from nice properties of fiber spaces, such as separability and reflexivity. It is also known that in practice, such equations possess a regularisation property that leads to continuous solutions.  
Pullback technique allows us to transfer the Oseledets decomposition to more regular spaces (with finer topology).

\section{Preliminaries and definitions} 

\smallskip\par
In this section, we will present definitions of measurable dynamical systems, measurable linear skew-product semidynamical systems, and the Oseledets decomposition. 

\par\smallskip

\subsection{Measurable dynamical systems}\label{subsec:MDS}

We write $\RR^{+}$ for $[0, \infty)$. For a metric space $S$ by $\mathfrak{B}(S)$ we denote the $\sigma$\nobreakdash-\hspace{0pt}algebra of Borel subsets of $S$. A probability space is a triple $\OFP$, where $\Omega$ is a set, $\mathfrak{F}$ is a $\sigma$\nobreakdash-\hspace{0pt}algebra of subsets of $\Omega$, and $\PP$ is a probability measure defined for all $F \in \mathfrak{F}$.  We always assume that the measure $\PP$ is complete.
\par\smallskip
A {\em measurable dynamical system} on the probability space $\OFP$ is a $(\mathfrak{B}(\RR) \otimes \mathfrak{F},\mathfrak{F})$\nobreakdash-\hspace{0pt}measurable mapping $\theta\colon\RR\times \Omega\to \Omega$ such that
\begin{itemize}[label=\raisebox{0.25ex}{\tiny$\bullet$}]
\item $\theta(0,\omega)=\omega$ for any $\omega  \in \Omega$,
\item $\theta(t+s,w)=\theta(t,\theta(s,\omega))$ for any $\omega  \in \Omega$ and $t,\,s \in \RR^{+}$.  
\end{itemize}
We write $\theta(t,\omega)$ as $\theta_t\omega$. Also, we usually denote measurable dynamical systems by $(\OFP,(\theta_{t})_{t \in \RR})$ or simply by $(\theta_{t})_{t \in \RR}$.\par
A {\em metric dynamical system} is a measurable dynamical system $(\OFP,(\theta_{t})_{t \in \RR})$
such that for each $t\in\RR$  the mapping $\theta_t\colon \Omega\to\Omega$ is $\PP$-preserving (i.e., $\PP(\theta_t^{-1}(F))=\PP(F)$ for any $F\in\mathfrak{F}$ and $t\in\RR$).
A subset $\Omega'\subset\Omega$ is \emph{invariant} if $\theta_t(\Omega')=\Omega'$ for all $t\in\RR$, and the metric dynamical system is said to be \emph{ergodic} if for any invariant subset  $F \in \mathfrak{F}$, either $\PP(F) = 1$ or $\PP(F) = 0$. Throughout the paper we will assume that $\PP$ is ergodic.

\subsection{Measurable linear skew-product semidynamical systems}\label{subsec:MLSPS} 

By a {\em measurable linear skew-product semidynamical system} or {\em semiflow}, 
$\Phi = \allowbreak ((U_\omega(t))_{\omega \in \Omega, t \in \RR^{+}}, \allowbreak (\theta_t)_{t\in\RR})$ on  $X$ covering a metric dynamical system $(\theta_{t})_{t \in \RR}$ we understand a $(\mathfrak{B}(\RR^{+}) \otimes \mathfrak{F} \otimes \mathfrak{B}(X), \mathfrak{B}(X))$\nobreakdash-\hspace{0pt}measurable
mapping
\begin{equation*}
[\, \RR^{+} \times \Omega \times X \ni (t,\omega,u) \mapsto U_{\omega}(t)\,u \in X \,]
\end{equation*}
satisfying
\begin{align}
&U_{\omega}(0) = \mathrm{Id}_{X} \quad & \textrm{for each }\,\omega  \in \Omega, \nonumber 
\\
&U_{\theta_{s}\omega}(t) \circ U_{\omega}(s)= U_{\omega}(t+s) \qquad &\textrm{for each } \,\omega \in \Omega \textrm{ and }  t,\,s \in \RR^{+},
\label{eq-cocycle}
\\
&[\, X \ni u \mapsto U_{\omega}(t)u \in X \,] \in \mathcal{L}(X) & \textrm{for each }\,\omega \in \Omega \textrm{ and } t \in \RR^{+}.\nonumber
\end{align}
Equation~\eqref{eq-cocycle} is called the cocycle property.

By the {\em positive semiorbit} passing through $(\omega,u)\in \Omega\times X$ we understand a $(\mathfrak{B}([0,\infty)),
\allowbreak\mathfrak{B}(X))$\nobreakdash-\hspace{0pt}measurable mapping 
\begin{equation*}
    \big[[0,\infty)\ni t \mapsto U_{\omega}(t)u\in X\big].
\end{equation*}
A {\em negative semiorbit} passing through $(\omega,u)\in \Omega\times X$ is a $(\mathfrak{B}((-\infty,0]),
\mathfrak{B}(X))$\nobreakdash-\hspace{0pt}measurable mapping $\tilde{u} \colon (-\infty,0] \to X$ such that
\begin{itemize}[label=\raisebox{0.25ex}{\tiny$\bullet$}]
\item  $\tilde{u}(0)=u$;
\item  $\tilde{u}(s+t)=U_{\theta_{s}\omega}(t)\tilde{u} (s)$ for each  $s \le 0$, $t\geq 0$  such that  $s+t\le 0$.
\end{itemize}
By a {\em full} or {\em entire} orbit passing through $(\omega,u)\in \Omega\times X$ we understand a $(\mathfrak{B}(\RR),
\allowbreak\mathfrak{B}(X))$\nobreakdash-\hspace{0pt}measurable mapping $\title{u} \colon \RR \to X$ such that
\begin{itemize}[label=\raisebox{0.25ex}{\tiny$\bullet$}]
    \item $\tilde{u}(0)=u$; 
    \item $\tilde{u}(s+t)=U_{\theta_{s}\omega}\tilde{u} (s)$ for each  $s \in\RR$ and $t\geq 0$.
\end{itemize}
From now on, we will focus on separable Banach spaces. Furthermore, based on the results from~\cite[Lemma 5.6 and Corollary 7.3]{Lian-Lu}, we will refrain from discussing measurability in the Grassmanian sense in favor of an equivalent definition of a measurable basis. Let $\Omega_0 \in \mathfrak{F}$.
A family $\{E(\omega)\}_{\omega \in \Omega_0}$ of $l$\nobreakdash-\hspace{0pt}dimensional vector subspaces of $X$ is {\em measurable\/} if there are $(\mathfrak{F},
\mathfrak{B}(X))$\nobreakdash-\hspace{0pt}measurable functions $v_1, \dots, v_l \colon \Omega_0 \to X$ such that $(v_1(\omega), \dots, v_l(\omega))$ forms a basis of $E(\omega)$ for each $\omega \in \Omega_0$.
\par\smallskip
Let $\{E(\omega)\}_{\omega \in \Omega_0}$ be a family of $l$\nobreakdash-\hspace{0pt}dimensional vector subspaces of $X$, and let $\{F(\omega)\}_{\omega \in \Omega_0}$ be a family of $l$\nobreakdash-\hspace{0pt}codimensional closed vector subspaces of $X$, such that $E(\omega) \oplus F(\omega) = X$ for all $\omega \in \Omega_0$.  We define the {\em family of projections associated with the decomposition\/} $E(\omega) \oplus F(\omega) = X$ as $\{P(\omega)\}_{\omega \in \Omega_0}$, where $P(\omega)$ is the linear projection of $X$ onto $F(\omega)$ along $E(\omega)$, for each $\omega \in \Omega_0$.
\par\smallskip
The family of projections associated with the decomposition $E(\omega) \oplus F(\omega) = X$ is called {\em strongly measurable\/} if for each $u \in X$ the mapping $[\, \Omega_0 \ni \omega \mapsto P(\omega)u \in X \,]$ is $(\mathfrak{F}, \mathfrak{B}(X))$\nobreakdash-\hspace{0pt}measurable.
\par\smallskip
We say that the decomposition $E(\omega) \oplus F(\omega) = X$, with
$\{E(\omega)\}_{\omega \in \Omega_0}$ finite\nobreakdash-\hspace{0pt}dimensional, is {\em invariant\/} if $\Omega_0$ is invariant, $U_{\omega}(t)E(\omega) = E(\theta_{t}\omega)$
and $U_{\omega}(t)F(\omega) \subset F(\theta_{t}\omega)$, for each $t \in \RR^+$.
\par\smallskip
A strongly measurable family of projections associated with the invariant decomposition $E(\omega) \oplus F(\omega) = X$ is referred to as {\em tempered\/} if
\begin{equation*} 
\lim\limits_{t \to \pm\infty} \frac{\ln{\norm{P(\theta_{t}\omega)}}}{t} = 0 \qquad \PP\text{-a.e. on }\Omega_0.
\end{equation*}

\subsection{Oseledets decomposition} 
From now on we assume that for a given semiflow 
\begin{enumerate}[label=\textbf{(E\arabic*)}, ref=$\mathrm{E\arabic*}$]
\item\label{a:Lim}  the functions
\begin{itemize}
    \item[] $\bigl[ \,\Omega\ni \omega \mapsto \sup\limits_{0 \le s \le 1} {\lnplus{\norm{U_{\omega}(s)}}} \in \RR^+  \, \bigr] \in L_1\OFP$, 
    \item[] $\bigl[ \, \Omega\ni \omega \mapsto \sup\limits_{0 \le s \le 1}
{\lnplus{\norm{U_{\theta_{s}\omega}(1-s)}}} \in \RR^+ \, \bigr]\in L_1\OFP$. 
\end{itemize}
\end{enumerate}
Then it follows from the Kingman subadditive ergodic theorem that there exists $\lambda_{\mathrm{top}} \in [-\infty, \infty)$ such that
\begin{equation*}
\lim\limits_{t \to \infty} \frac{\ln{\norm{U_{\omega}(t)}}}{t} = \lambda_{\mathrm{top}}
\end{equation*}
for $\PP$-a.e.\ $\omega \in \Omega$, which is referred to as the {\em top Lyapunov exponent\/} of $\Phi$. 

\begin{enumerate}[resume,label=\textbf{(E\arabic*)}, ref=$\mathrm{L\arabic*}$]
\item\label{a:Lap} $\lambda_{\mathrm{top}} > -\infty$.
\end{enumerate}


\begin{definition}[Oseledets decomposition]
$\Phi$ admits an Oseledets decomposition if there exists an invariant subset $\Omega_0 \subset \Omega$, $\PP(\Omega_0) = 1$, with the property that one of the following mutually exclusive cases, \eqref{Osel:case_I} or \eqref{Osel:case_II}, holds:\par
\begin{enumerate}
[label=\textup{\textbf{(O\arabic*)}},ref=O\arabic*]
\item\label{Osel:case_I}
There are $k$  real numbers $\lambda_1 = \lambda_{\mathrm{top}} > \cdots > \lambda_k$, called  the {\em Lyapunov exponents\/} for $\Phi$, $k$ measurable families $\{E_1(\omega)\}_{\omega\in\Omega_0}$, \ldots, $\{E_k(\omega)\}_{\omega \in \Omega_0}$ of finite dimensional vector subspaces, and a family $\{F_{\infty}(\omega)\}_{\omega\in \Omega_0}$ of closed vector subspaces of finite codimension such that\par\smallskip
\begin{enumerate}[label=\textup{(}\textit{\roman*}\textup{)}, ref=\textit{\roman*},leftmargin=5pt]
\item\label{O1:1} for $j = 1,\dots,k$, any $\omega\in\Omega_0$ and $t \ge 0$
\begin{equation*}
    U_{\omega}(t) E_{j}(\omega) = E_{j}(\theta_t\omega )  \quad \& \quad  U_{\omega}(t) F_{\infty}(\omega) \subset F_{\infty}(\theta_t\omega);
\end{equation*}
\item\label{O1:2}
$E_1(\omega) \oplus \ldots \oplus E_k(\omega) \oplus F_{\infty}(\omega) = X$ for any $\omega \in \Omega_0$; we write
\begin{equation*}
     F_j(\omega) \vcentcolon= \bigoplus\limits_{m=j+1}^{k} E_{m}(\omega) \oplus F_{\infty}(\omega) \text{ for } j = 0, \dots, k.
\end{equation*}
In particular, $F_j(\omega) = E_{j+1}(\omega) \oplus F_{j+1}(\omega)$ for $j = 0, 1, \dots, k - 2\,$;
\item\label{O1:3}
for $j = 1,\dots,k$, the families of projections associated with the decomposition
\begin{equation*}
    \Bigl(\bigoplus\limits_{n=1}^{j} E_{n}(\omega)\Bigr) \oplus F_j(\omega) = X
\end{equation*}
is strongly measurable and tempered;
\item\label{O1:4} for $j = 1,\dots, k\,$, any $\omega \in \Omega_0$ and any nonzero $u \in E_j(\omega)$
\begin{equation*}
  \lim_{t\to \infty} \frac{\ln{\norm{U_{\omega}(t)\,u}}}{t}  = \lambda_j;
\end{equation*}
\item\label{O1:7} for $j = 1,\dots, k$ and any $\omega \in \Omega_0$, a nonzero $u \in F_{j - 1}(\omega)$ belongs to $E_{j}(\omega)$ if and only if there exists a negative semiorbit $\tilde{u} \colon (-\infty ,0] \to X$ passing through $(\omega, u)$ such that
\begin{equation*}
\lim\limits_{s \to -\infty} \frac{\ln{\norm{\tilde{u}(s)}}}{s}  = \lambda_{j};
\end{equation*}
\item\label{O1:8} for any $\omega \in \Omega_0$
\begin{equation*}
    \lim\limits_{t\to\infty} \frac{\ln{\norm{U_{\omega}(t){\restriction}_{F_{\infty}(\omega)}}} }{t} = -\infty.
\end{equation*}
\end{enumerate}
In this case, $\{F_1(\omega)\}_{\omega \in \Omega_0}, \ldots, \{F_{k-1}(\omega)\}_{\omega \in \Omega_0}, \{F_{\infty}(\omega)\}_{\omega \in \Omega_0}$ is called the {\em Oseledets filtration\/} for $\Phi$.
\item\label{Osel:case_II}
There is a decreasing sequence of real numbers $\lambda_1 = \lambda_{\mathrm{top}}> \cdots > \lambda_j > \lambda_{j+1} > \cdots {}$ with limit $-\infty$, called  the  Lyapunov exponents for $\Phi$, countably many measurable families $\{E_j(\omega)\}_{\omega \in
\Omega_0}$, $j\in\NN$, of finite dimensional vector subspaces, and countably many families $\{F_j(\omega)\}_{\omega \in \Omega_0}$, $j\in\NN$,
 of closed vector subspaces of finite codimensions, called the {\em Oseledets filtration\/} for $\Phi$,  such that 
\begin{enumerate}[label=\textup{(}\textit{\roman*}\textup{)}, ref=\textit{\roman*}, leftmargin=5pt]
\item\label{O2:1} for $j\in\NN$, any $\omega \in \Omega_0$ and $t \ge 0$
\begin{equation*}
    U_{\omega}(t) E_{j}(\omega) = E_{j}(\theta_t\omega)\quad \& \quad U_{\omega}(t) F_{j}(\omega) \subset F_{j}(\theta_t\omega);
\end{equation*}
\item\label{O2:2} for $j\in\NN$ and any $\omega \in \Omega_0$
\begin{equation*}
    E_1(\omega) \oplus \ldots \oplus E_j(\omega) \oplus F_j(\omega) = X \quad\&\quad F_j(\omega) = E_{j+1}(\omega) \oplus F_{j+1}(\omega); 
\end{equation*}
\item\label{O2:3} for $j\in\NN$,
the families of projections associated with the decompositions
\begin{equation*}
    \Bigl(\bigoplus\limits_{n=1}^{j} E_{n}(\omega)\Bigr) \oplus F_{j}(\omega) = X
\end{equation*}
are strongly measurable and tempered;
\item\label{O2:4} for $j\in\NN$,  any $\omega \in \Omega_0$ and any nonzero $u \in E_j(\omega)$
\begin{equation*}
     \lim_{t \to \infty} \frac{\ln{\norm{U_{\omega}(t)\,u}}}{t}  = \lambda_j;
\end{equation*}
\item\label{O2:6} for $j\in\NN$ and any $\omega \in \Omega_0$, a nonzero $u \in F_{j - 1}(\omega)$ belongs to $E_j(\omega)$ if and only if there exists a negative semiorbit $\tilde{u} \colon (-\infty ,0] \to X$ passing through $(\omega, u)$ such that 
\begin{equation*}
\lim\limits_{s \to -\infty} \frac{\ln{\norm{\tilde{u}(s)}}}{s}  = \lambda_{j};
\end{equation*}
\item\label{O2:7} for $j\in\NN$ and any $\omega \in \Omega_0$
\begin{equation*}
    \lim\limits_{t\to\infty} \frac{\ln{\norm{U_{\omega}(t){\restriction}_{F_j(\omega)}}}}{t}  = \lambda_{j+1};
\end{equation*}
\end{enumerate}
\end{enumerate}
\end{definition}

\section{Sub-semiflows}

Let $\Phi^{(1)} = \allowbreak ((U^{(1)}_\omega(t))_{\omega \in \Omega, t \in \RR^{+}}, (\theta_t)_{t\in\RR})$ and $\Phi^{(2)} = \allowbreak ((U^{(2)}_\omega(t))_{\omega \in \Omega, t \in \RR^{+}}, \allowbreak (\theta_t)_{t\in\RR})$ be measurable linear skew-product semidynamical systems on Banach spaces \linebreak $(X_1, \norm{\cdot}_1)$ and $(X_2,  \norm{\cdot}_2)$ respectively.  The operator norm on $\mathcal{L}(X_m, X_n)$, $m, n  = 1, 2$, will be denoted by $\norm{\cdot}_{m, n}$.   By $\mathcal{L}_{\mathrm{s}}(X_m,X_n)$ we understood the space of linear bounded operators equipped with strong operator topology.

We assume: 

\begin{enumerate}
[label=$\mathrm{\textbf{(A\arabic*)}}$,ref=$\mathrm{A\arabic*}$] 
\item\label{a:emdeding} There is an injective bounded linear map $i\vcentcolon X_{2} \to X_{1}$.
\item \label{a:commutation}  For any $\omega\in\Omega$ and $t\ge 0$ the equality $U^{(1)}_{\omega} (t) \circ i = i\circ U^{(2)}_{\omega}(t)$ holds.
\item\label{a:connector} For any $\omega\in\Omega$ there is a map $U^{(1,2)}_{\omega}(1)\in\mathcal{L}(X_1,X_2)$ such that $U^{(1)}_{\omega}(1)=i\circ U^{(1,2)}_{\omega}(1)$ and $U^{(2)}_{\omega}(1)=U^{(1,2)}_{\omega}(1) \circ i$. Moreover, the operator norm bound $\norm{U^{(1,2)}_{\omega}(1)}_{1, 2}$ is uniform on $\omega$, i.e., there exists $M>0$ such that for all $\omega\in\Omega$ the inequality $\norm{U^{(1,2)}_{\omega}(1)}_{1, 2} \le M$ holds. 
\item\label{a:U_measurable} The operator 
is $U^{(1,2)}_{\omega}(1)$ is $(\mathfrak{F}, \mathfrak{B}(\mathcal{L}_{\mathrm{s}}(X_1,X_2)))$\nobreakdash-\hspace{0pt}measurable i.e., for all $u\in X_1$ the mapping \begin{equation*}
    \big[\,\Omega\ni \omega\mapsto U^{(1,2)}_{\omega}(1)u\in X_2 \,\big]   \text{ is } (\mathfrak{F}, \mathfrak{B}(X_2))\text{\nobreakdash-\hspace{0pt}measurable}. 
\end{equation*}
\item\label{a:separable} $X_1$ and $X_2$ are separable.
\item\label{a:sigma-compatibility} For any $A\in\mathfrak{B}(X_2)$ there exists $B\in \mathfrak{B}(X_1)$ such that $A=i^{-1}(B)$.
\end{enumerate}

\begin{lemma}
\label{lemma:strong_measurable_paring}
   For a separable Banach space $X_1$ the mapping 
    \begin{equation*}
        \big[ \mathcal{L}_{\mathrm{s}}(X_1,X_2) \times X_1 \ni (T,x)\mapsto Tx\in X_2  \big]
    \end{equation*}
    $(\mathfrak{B}(\mathcal{L}_{\mathrm{s}}(X_1,X_2))\otimes\mathfrak{B}(X_1), \mathfrak{B}(X_2) )$\nobreakdash-\hspace{0pt}measurable.
\end{lemma}
\begin{proof}
    Compare with~\cite[Lemma A.6 (2)]{GTQu} and~\cite[Lemma. 6.4.2(i)]{V.B.}. 
\end{proof} 

\begin{lemma}
\label{lemma:U_(1,2)}
 Assume \eqref{a:emdeding}, \eqref{a:commutation} and \eqref{a:connector}. For any $\omega\in\Omega$ and $t\ge 1$ the equality 
\begin{equation*}
    U^{(1,2)}_{\theta_{t-1}\omega}(1) \circ U^{(1)}_{\omega}(t-1) = U^{(2)}_{\theta_{1}\omega}(t-1) \circ U^{(1,2)}_{\omega}(1) 
\end{equation*}
holds.  
\end{lemma}
\begin{proof}
   Fix $\omega$ and $t$. We are going to use the assumptions~\eqref{a:commutation} and \eqref{a:connector}. Moreover, since $i$ is injective it suffices to observe that  
\begin{align*}
        i \circ  U^{(1,2)}_{\theta_{t-1}\omega}(1) \circ U^{(1)}_{\omega}(t-1) &= U^{(1)}_{\theta_{t-1}\omega}(1) \circ U^{(1)}_{\omega}(t-1) \\
        & = U^{(1)}_{\theta_1\omega}(t-1) \circ U^{(1)}_{\omega}(1) \\  
        &=  U^{(1)}_{\theta_1\omega}(t-1) \circ i\circ U^{(1,2)}_{\omega}(1) \\
        &= i\circ U^{(2)}_{\theta_{1}\omega}(t-1) \circ U^{(1,2)}_{\omega}(1).  
\end{align*}
\end{proof}

The above observation allows us to extend the definition of the operator $U^{(1,2)}_{\omega}(1)$ to all $t\ge 1$ as 
\begin{equation*}
   U^{(1,2)}_{\omega}(t)  \vcentcolon =  U^{(1,2)}_{\theta_{t-1}\omega}(1) \circ U^{(1)}_{\omega}(t-1) = U^{(2)}_{\theta_{1}\omega}(t-1) \circ U^{(1,2)}_{\omega}(1)\in\mathcal{L}(X_2,X_1). 
\end{equation*}

Before we investigate how to transfer Oseledets decomposition from one measurable linear skew-product semidynamical systems $\Phi^{(1)}$ onto another semiflow $\Phi^{(2)}$, we will introduce a useful Claim. 

\begin{claim}\label{claim:E_2-subspace}
Assume~\eqref{a:emdeding}, \eqref{a:commutation}, \eqref{a:connector}  and let $\Omega_0\in \mathfrak{F}$ be such that $\PP(\Omega_0)=1$.  For any family of subspaces $\{W(\omega)\}_{\omega\in \Omega_0}$ of $X_1$ such that the equality $U^{(1)}_{\omega}(t)W(\omega)=W(\theta_{t}\omega)$ holds for all $t\ge 0$ and all $\omega\in\Omega_0$ there exists a family of subspaces $\{V(\omega)\}_{\omega\in \Omega_0}$ of $X_2$ such 
\begin{enumerate}
[label=\textup{(}\roman*\textup{)},ref=\textit{\roman*}]
\item\label{claim:sub-space-i} $i V(\omega) =W(\omega)$ for any $\omega\in\Omega_0$,
\item\label{claim:sub-space-ii} $U^{(2)}_{\omega}(t)V(\omega)=V(\theta_{t}\omega)$ for all $t\ge 0$ and $\omega\in\Omega_0$.
\end{enumerate}
\end{claim}
\begin{proof}
For $\omega\in\Omega_0$ let $V(\omega) \vcentcolon=U^{(1,2)}_{\theta_{-1}\omega} (1) W(\theta_{-1}\omega)$. Therefore,
\begin{align*}
     iV(\omega) =i\circ U^{(1,2)}_{\theta_{-1}\omega}(1)W(\theta_{-1}\omega)&=   U^{(1)}_{\theta_{-1}\omega}(1) W(\theta_{-1}\omega) = W( \omega),
\end{align*}
and the part~\eqref{claim:sub-space-i} is done. Furthermore, for fixed $t\ge 0$ we have 
\begin{align*}
   i \circ U^{(2)}_{\omega}(t)V(\omega) &= i \circ U^{(2)}_{\omega}(t)\circ  U^{(1,2)}_{\theta_{-1}\omega}(1) W(\theta_{-1}\omega)\\
     & =  U^{(1)}_{\omega}(t) \circ i \circ  U^{(1,2)}_{\theta_{-1}\omega}(1) W(\theta_{-1}\omega) \\
     &=  U^{(1)}_{\omega}(t) \circ  U^{(1)}_{\theta_{-1}\omega}(1)W(\theta_{-1}\omega)\\
     &=  U^{(1)}_{\theta_{t-1}\omega}(1) \circ  U^{(1)}_{\theta_{-1}\omega}(t) W(\theta_{-1}\omega)\\
     &=  U^{(1)}_{\theta_{t-1}\omega}(1)  W(\theta_{t-1}\omega)\\
     &=i\circ U^{(1,2)}_{\theta_{t-1}\omega}(1)  W(\theta_{t-1}\omega)\\
     &= i\circ U^{(1,2)}_{\theta_{-1} \theta_{t}\omega}(1)  W(\theta_{-1} \theta_{t}\omega) \\
     &=i V(\theta_t\omega), 
\end{align*}
which concludes the proof. 
\end{proof}
It should be noted that $V(\omega)$ can be defined as $i^{-1}(W(\omega))$. In such a case, firstly, we can observe that by invariance of $W(\omega)$ and the assumption~\eqref{a:connector} we have  
\begin{equation*}
    W(\theta_t \omega) = U^{(1)}_{\theta_{t-1}\omega}(1)  W(\theta_{t-1}\omega) = i \circ U^{(1,2)}_{\omega}(1) W(\theta_{t-1}\omega)\subset iX_2.  
\end{equation*}   
Hence, by proceeding as in the proof of Claim~\ref{claim:E_2-subspace} we can show invariance of $V(\omega)$. Indeed,  
\begin{align*}
   i \circ U^{(2)}_{\omega}(t)V(\omega) &= i \circ U^{(2)}_{\omega}(t)i^{-1}(W(\omega))\\
     & =  U^{(1)}_{\omega}(t) \circ i (i^{-1}(W(\omega))) \\
     &=   U^{(1)}_{\omega}(t) (W(\omega)  \cap iX_2) \\ 
     & = U^{(1)}_{\omega}(t) W(\omega)  \\
     & =  W(\theta_t \omega) \\
     & =  W(\theta_t \omega) \cap i X_2  \\
     & = i(i^{-1} W(\theta_t \omega)). 
\end{align*}

Assume that $\Phi^{(1)}$ admits an Oseledets decomposition.  By taking in Claim~\ref{claim:E_2-subspace} $\{E^{(1)}_j(\omega)\}_{\omega\in\Omega_0,j=1,\dots,k}$ or $\{E^{(1)}_j(\omega)\}_{\omega\in\Omega_0,j\in \NN}$ for $\{W(\omega)\}$ we obtain families \linebreak $\{E^{(2)}_j(\omega)\}_{\omega\in\Omega_0,j=1,\dots,k}$ or $\{E^{(2)}_j(\omega)\}_{\omega\in\Omega_0,j\in \NN}$ of invariant, finite-dimensional vector subspaces of $X_2$.  

Fix $j$.  Let $l$ be the dimension of $E^{(1)}_j$.  For $\omega \in \Omega_0$ denote by $G(\omega)$ the linear isomorphism from $E^{(1)}_j(\omega)$ onto $\RR^{l}$ given by $G(\omega)u \vcentcolon= (\alpha_1, \ldots, \alpha_l)$ where $u = \alpha_1 \,v_1 (\omega)+ \ldots + \alpha_l\, v_l(\omega)$ is written in the basis.  The family $\{ \widehat{U}_{\omega}(1) \}_{\omega \in \Omega_0}$ of linear automorphisms of $\RR^l$ defined by
    \begin{equation*}
        \widehat{U}_{\omega}(1) \vcentcolon= G(\theta_{1}\omega) \circ U^{(1)}_{\omega}(1){\restriction}_{E^{(1)}_j(\omega)} \circ G(\omega)^{-1}, \quad \omega \in \Omega_0,
    \end{equation*}
    generates a two\nobreakdash-\hspace{0pt}sided discrete\nobreakdash-\hspace{0pt}time linear skew\nobreakdash-\hspace{0pt}product dynamical system $\widehat{\Phi} = ((\widehat{U}_{\omega}(n)), (\theta_n))$ on $\Omega_0 \times \RR^l$, with
    \begin{equation*}
        \widehat{U}_{\omega}(n) \vcentcolon= G(\theta_{n}\omega) \circ U^{(1)}_{\omega}(n){\restriction}_{E^{(1)}_j(\omega)} \circ G(\omega)^{-1}, \quad \omega \in \Omega_0, \ n \in \NN.
    \end{equation*}
    This allows us to use results in~\cite[Subsection~4.2.4]{Vi}.  Indeed, those results are formulated for discrete time, but assumption \eqref{a:Lim} allows us to extend them to the continuous time case.  As a consequence of~\cite[Proposition~4.11]{Vi}, 

\begin{equation}
\label{eq:Lap_on_E}
         \lim\limits_{t\to\infty} \frac{1}{t}\ln{\norm{U^{(1)}_{\omega}(t){\restriction}_{E^{(1)}_{j}(\omega) }}_{1, 1}}  = \lambda_j \text{ and } \lim\limits_{t\to\infty} \frac{1}{t}\ln{\norm{(U^{(1)}_{\omega}(t){\restriction}_ {E^{(1)}_{j}(\omega)})^{-1}  }^{-1}_{1, 1}}  = \lambda_j.
\end{equation} 
Based on the above, we can introduce the following lemma, which will prove to be useful in the subsequent analysis. 
\begin{lemma}
\label{lemma:Lap-Projection-estim}
    For $\PP$-a.e. $\omega\in\Omega$ and each $j$, there exists a function $c:(1,\infty)\to (0,\infty)$ such 
\begin{enumerate}
[label=\textup{(}\roman*\textup{)},ref=\textit{\roman*}]
\item\label{lemma:Lap-Projection-estim-i} for each $t\ge 1$ and $u\in E^{(2)}_j(\theta_t\omega)$ we have $c(t)\norm{u}_2\le \norm{iu}_1$, 
\item\label{lemma:Lap-Projection-estim-ii} $c$ is sub-exponential i.e. 
\begin{equation*}
    \lim_{t\to \infty} \frac{\ln c(t)}{t} =0.
\end{equation*}
\end{enumerate}
\end{lemma} 
\begin{proof} Let, 
\begin{equation*}
  c(t)=\inf \Big\{\frac{\norm{iu'}_1 }{\norm{u'}_2} : u'\in E^{(2)}_j(\theta_t\omega)\setminus\{0\} \Big\} 
\end{equation*}
Hence, part~\eqref{lemma:Lap-Projection-estim-i} is trivially satisfied.  
  For part~\eqref{lemma:Lap-Projection-estim-ii} observe that the upper bound for $c$ is obviously $\norm{i}$. Therefore, we are concentrating on the sub-exponential lower limit. It is claimed that  
    \begin{equation*}
        \frac{ \norm{(U^{(1)}_{\omega}(t){\restriction}_{E^{(1)}_{j}(\omega)})^{-1}  }^{-1}_{1, 1}} {\norm{U^{(1,2)}_{\theta_{t-1}\omega}(1)}_{1, 2} \, \norm{U^{(1)}_{\omega}(t-1){\restriction}_{E^{(1)}_{j}(\omega)} }_{1, 1} }\le c(t), \quad \text{for } t\ge 1.
    \end{equation*}
    Indeed, fix $t\ge 1$ and $u'\in E_j^{(2)}(\theta_t\omega)\setminus\{0\}$. Let $u\in E_j^{(1)}(\omega)$ be such that $u'=U^{(1,2)}_{\omega}(t)u$. Since 
    \begin{align*}
        & \norm{u'}_2 = \norm{U^{(1,2)}_{\theta_{t-1}\omega}(1) U^{(1)}_{\omega}(t-1)u }_2 \le \norm{U^{(1,2)}_{\theta_{t-1}\omega}(1)}_{1, 2} \, \norm{U^{(1)}_{\omega}(t-1){\restriction}_{E^{(1)}_{j}(\omega)} }_{1, 1} \, \norm{u}_1,  \\
        & \norm{iu'}_1= \norm{ U^{(1)}_{\omega}(t) u }_{1, 1} \ge \norm{(U^{(1)}_{\omega}(t){\restriction}_{E^{(1)}_{j}(\omega)})^{-1}  }_{1, 1}^{-1} \, \norm{u}_1.
    \end{align*}
We thus have further
\begin{equation*}
    \frac{ \norm{(U^{(1)}_{\omega}(t){\restriction}_{E^{(1)}_{j}(\omega)})^{-1}  }^{-1}_{1, 1}} {\norm{U^{(1,2)}_{\theta_{t-1}\omega}(1)}_{1, 2} \, \norm{U^{(1)}_{\omega}(t-1){\restriction}_{E^{(1)}_{j}(\omega)} }_{1, 1} } \le \frac{\norm{iu'}_1}{\norm{u'}_2}.
\end{equation*}
From this the preliminary claim is obtained, which makes it possible to write the following
\begin{gather*}
    \frac{ 1 }{t} \ln \norm{(U^{(1)}_{\omega}(t){\restriction}_{E^{(1)}_{j}(\omega)})^{-1}  }^{-1}_{1,1} - \frac{ 1}{t }  \ln \norm{U^{(1)}_{\omega}(t-1){\restriction}_{E^{(1)}_{j}(\omega)} }_{1,1} - \frac{1}{t} \ln \norm{U^{(1,2)}_{\theta_{t-1}\omega}(1)}_{1,2}   \\ 
    \le \frac{\ln c(t)}{t} \le \frac{\ln \norm{i}_{2, 1}}{t}
\end{gather*}
Equation~\eqref{eq:Lap_on_E} concludes the proof.
\end{proof} 

\begin{theorem}
\label{thm:main}
    Assume~\eqref{a:emdeding}-\eqref{a:sigma-compatibility} and moreover, that $\Phi^{(1)}$ admits an Oseledets decomposition in case~\eqref{Osel:case_I} \textup{(}or~\eqref{Osel:case_II}\textup{)}. Then $\Phi^{(2)}$ admits an Oseledets decomposition in case~\eqref{Osel:case_I} \textup{(}or~\eqref{Osel:case_II} respectively\textup{)}.  
\end{theorem}
\begin{proof}
    We are going to show that if $\Phi^{(1)}$ admits an Oseledets decomposition in case~\eqref{Osel:case_I} with parameters  
\begin{equation*}
     (\Omega_0, (\lambda_j)_{j=1,\dots,k}, \{ E_j(\omega)\}_{\omega\in\Omega_0,j=1,\dots,k}, \{F_{\infty}(\omega)\}_{\omega \in \Omega_0}) 
 \end{equation*}
    then $\Phi^{(2)}$ admits an Oseledets decomposition in case~\eqref{Osel:case_I} with parameters 
  \begin{equation*}
     (\Omega_0, (\lambda_j)_{j=1,\dots,k}, \{U^{(1,2)}_{\theta_{-1}\omega}(1) E_j(\theta_{-1}\omega)\}_{\omega\in\Omega_0,j=1,\dots,k}, \{i^{-1}(F_{\infty}(\omega))\}_{\omega \in \Omega_0}). 
 \end{equation*}
 Before we start with cases~\eqref{O1:1}-\eqref{O1:8}. Note that $i^{-1}(F_{\infty}(\omega))$ is a closed subspace as a continuous preimage of a closed subspace. Moreover, for each $j=1,\dots,k$ the family $\{U^{(1,2)}_{\theta_{-1}\omega}(1) E_j(\theta_{-1}\omega)\}_{\omega\in\Omega_0}$ is measurable. Indeed, the mapping 
 \begin{equation*}
   \big[ \, E_{j}(\omega) \ni u \mapsto    U^{(1,2)}_{\theta_{-1}\omega}(1)u  \in U^{(1,2)}_{\theta_{-1}\omega}(1) E_j(\theta_{-1}\omega)\, \big]
 \end{equation*}
 is a bijection. It follows from the observation that $U^{(1)}_{\omega}(1)$ is a bijection since it preserves finite dimension i.e., $\dim U^{(1)}_{\omega}(1) E_{j}(\omega) = \dim E_{j}(\theta_{1}\omega)$ (see~\eqref{Osel:case_I}\eqref{O1:1}). Therefore, from: the injectivity of $i$, the bijectivity of $U^{(1)}_{\omega}(1)$ and $i\circ U^{(1,2)}_{\omega}(1)=U^{(1)}_{\omega}(1)$ we can deduce that $u=v$ whenever, $i\circ U^{(1)}_{\theta_{-1}\omega}(1) u = i\circ U^{(1)}_{\theta_{-1}\omega}(1) v$. Moreover, we have 
\begin{align*}
     U^{(1,2)}_{\theta_{-1}\omega}(1) E_j(\theta_{-1}\omega) & = \spanned \{ U^{(1,2)}_{\theta_{-1}\omega}(1) v_1(\theta_{-1}\omega),\dots,U^{(1,2)}_{\theta_{-1}\omega}(1) v_l(\theta_{-1}\omega)  \}, 
 \end{align*}
where $v_1, \dots, v_l \colon \Omega_0 \to X$ are $(\mathfrak{F},
\mathfrak{B}(X))$\nobreakdash-\hspace{0pt}measurable functions, such that $ (v_1(\omega), \allowbreak \dots, v_l(\omega))$ forms a basis of $E_j(\omega)$ for each $\omega \in \Omega_0$. Hence, it remains to show the $(\mathfrak{F},
\mathfrak{B}(X_2))$\nobreakdash-\hspace{0pt}measurability of the maps  
\begin{equation*}
   \big[ \, \Omega_0 \ni \omega \mapsto U^{(1,2)}_{\theta_{-1}\omega}(1) v_j(\theta_{-1}\omega) \in X_2\, \big], 
\end{equation*} 
for $j=1,\dots,l$. It can be done by looking at the composition 
\begin{equation*}
    \omega\mapsto (\omega,\omega)\mapsto (U_{\theta_{-1}\omega}^{(1,2)}(1),  v_j(\theta_{-1}\omega )) \mapsto U_{\theta_{-1}\omega}^{(1,2)}(1) v_j(\theta_{-1}\omega). 
\end{equation*}
The second mapping is $(\mathfrak{F}\otimes \mathfrak{F}, \mathfrak{B}(\mathcal{L}_{\mathrm{s}}(X_1,X_2))\otimes\mathfrak{B}(X_1))$\nobreakdash-\hspace{0pt}measurable. By the assumption~\eqref{a:separable} and Lemma~\ref{lemma:strong_measurable_paring}, the pairing is $(\mathfrak{B}(\mathcal{L}_{\mathrm{s}}(X_1,X_2))\otimes\mathfrak{B}(X_1), \allowbreak\mathfrak{B}(X_2) )$\nobreakdash-\hspace{0pt}measurable. 

\smallskip\par\noindent
\eqref{Osel:case_I}\eqref{O1:1} can be done via Claim~\ref{claim:E_2-subspace}, indeed
\begin{equation*}
    U^{(2)}_{\omega}(t) U^{(1,2)}_{\theta_{-1}\omega}(1) E_j(\theta_{-1}\omega) = U^{(1,2)}_{\theta_{t-1}\omega}(1) E_j(\theta_{t-1}\omega). 
\end{equation*}
It remains to show $U^{(2)}_{\omega}(t)i^{-1}(F_{\infty}(\omega))\subset i^{-1}(F_{\infty}(\theta_{t}\omega))$, equivalently
\begin{align*} 
    i\circ U^{(2)}_{\omega}(t) ( i^{-1}(F_{\infty}(\omega)))
    &=U^{(1)}_{\omega}(t) \circ i (i^{-1}(F_{\infty}(\omega)))\\
    &=U^{(1)}_{\omega}(t) (F_{\infty}(\omega)\cap iX_2)\\
    &\subset U^{(1)}_{\omega}(t) F_{\infty}(\omega) \cap U^{(1)}_{\omega}(t) i X_2 \\
    &=U^{(1)}_{\omega}(t) F_{\infty}(\omega) \cap  i  U^{(2)}_{\omega}(t) X_2
    \\
    &\subset F_{\infty}(\theta_{t}\omega) \cap  i X_2  \\
    &= i (i^{-1}(F_{\infty}(\theta_{t}\omega))). 
\end{align*}
Hence, the decomposition is invariant. 
\smallskip\par\noindent
\eqref{Osel:case_I}\eqref{O1:2} in order to show 
\begin{equation*}
    \bigoplus\limits_{j=1}^{k} U^{(1,2)}_{\theta_{-1}\omega}(1) E_j(\theta_{-1}\omega) \oplus i^{-1}(F_{\infty}(\omega)) =X_2 
\end{equation*}
it suffices to show that the spaces $U^{(1,2)}_{\theta_{-1}\omega}(1) E_j(\theta_{-1}\omega)$ ($j=1,\dots,k$) and $i^{-1}(F_{\infty}(\omega))$ are linearly independent and 
\begin{equation*}
    \sum\limits_{j=1}^{k} U^{(1,2)}_{\theta_{-1}\omega}(1) E_j(\theta_{-1}\omega) + i^{-1}(F_{\infty}(\omega)) =X_2. 
\end{equation*}
Therefore, fix $e^{(2)}_j\in U^{(1,2)}_{\theta_{-1}\omega}(1) E_j(\theta_{-1}\omega)$ for each $j=1,\dots,k$ and $f^{(2)}\in i^{-1}(F_{\infty}(\omega))$ such that $e^{(2)}_1+  e^{(2)}_2+\dots+ e^{(2)}_k+f^{(2)} = 0$. Since
\begin{equation*}
    0 = i \Big( \sum\limits_{j=1}^{k} e^{(2)}_j+f^{(2)} \Big) = \sum\limits_{j=1}^{k} ie^{(2)}_j+if^{(2)}, 
\end{equation*}  
and by linear independence of the subspaces  $E_j(\omega)$ ($j=1,\dots,k$) and $F_{\infty}(\omega)$ we have $ie^{(2)}_1=\dots=ie^{(2)}_k=if^{(2)} = 0$, so $e^{(2)}_1=\dots=e^{(2)}_k=f^{(2)} = 0$. It remains to show the sets equality. The inclusion $(\subset)$ is obvious. Therefore, fix $x\in X_2$. Hence, $ix\in X_1$ there exists a (unique) representation $ix=e^{(1)}_1+  e^{(1)}_2+\dots+ e^{(1)}_k+f^{(1)}$
for some $e^{(1)}_j\in E_{j}(\omega)$ and $f^{(1)}\in F_{\infty}(\omega)$. However, from the first part of the Claim~\ref{claim:E_2-subspace}\eqref{claim:sub-space-i} we can conclude that for any $e^{(1)}_j\in E_{j}(\omega)$ there exists $e^{(2)}_j\in U^{(1,2)}_{\theta_{-1}\omega}(1) E_j(\theta_{-1}\omega)$ such that $e^{(1)}_j=ie^{(2)}_j$. We can find $f^{(2)}\in F_{\infty}(\omega)$ such that $if^{(2)}=f^{(1)}$, namely $f^{(2)}= x-e^{(2)}_1-\dots-e^{(2)}_k$. So, the inclusion $(\supset)$ is done.  

\smallskip\par\noindent
\eqref{Osel:case_I}\eqref{O1:3} 
Fix $l=1,\dots,k$ and $u\in X_2$. In order to show that the family of projections associated with the decomposition
\begin{equation}
\label{eq:dec_X_2}
    \Bigl(\bigoplus\limits_{j=1}^{l} U^{(1,2)}_{\theta_{-1}\omega}(1) E_j(\theta_{-1}\omega) \Bigr) \oplus  \Bigl(  \bigoplus\limits_{j=l+1}^{k} U^{(1,2)}_{\theta_{-1}\omega}(1) E_j(\theta_{-1}\omega) \oplus i^{-1}(F_{\infty}(\omega)) \Bigr)  = X_2
\end{equation}
along the first component is strongly measurable it suffices to see that for fixed set $A\in\mathfrak{B}(X_2)$ we have 
\begin{align*}
    (P^{(2)}(\cdot)u )^{-1}(A) &= (P^{(2)}(\cdot)u )^{-1}(i^{-1}(B))\\
    &=(i\circ P^{(2)}(\cdot)u )^{-1}(B)\\
    &= ( P^{(1)}(\cdot) i u )^{-1}(B)\\
    &\in \mathfrak{F},
\end{align*}
where $P^{(2)}(\omega)$ is the projection along $U^{(1,2)}_{\theta_{-1}\omega}(1) E_1(\theta_{-1}\omega)\oplus\dots\oplus U^{(1,2)}_{\theta_{-1}\omega}(1) E_l(\theta_{-1}\omega) $ and set $B\in\mathfrak{B}(X_1)$ is the set from the assumption~\eqref{a:sigma-compatibility} such that $A=i^{-1}(B)$. It remains to show that the family of projections associated with the decomposition~\eqref{eq:dec_X_2} is tempered. The definition refers to projections onto the finite codimensional closed vector subspaces $F(\omega)$. However, it is a well-known fact that it is sufficient to equivalently demonstrate the temperedness of projections onto finite-dimensional vector subspaces (i.e. first component of~\eqref{eq:dec_X_2}, see~\cite[Remark 2.1]{Marek-Janusz-Sylvia-Rafa}).
We start by demonstrating that the projections onto $U^{(1,2)}_{\theta_{-1}\omega}(1) E_j(\theta_{-1}\omega)$ for any $j=1,\dots, k$ are tempered. Let us fix $j$. Subsequently, in the inductive process, we will show that temperedness holds universally. Therefore, let $\widetilde{P}_j^{(2)}(\omega)$ denote the complementary part of $P_j^{(2)}(\omega)$, namely the projection onto $U^{(1,2)}_{\theta_{-1}\omega}(1) E_j(\theta_{-1}\omega)$. Using the results from Lemma~\ref{lemma:Lap-Projection-estim}, we can conclude that for any $u\in X_2$ we have 
\begin{equation*}
    c(t) \norm{\widetilde{P}_j^{(2)}(\theta_t\omega) u }_2 \le \norm{i \widetilde{P}^{(2)}(\theta_t\omega) u}_1 \le \norm{i}_{2,1} \, \norm{\widetilde{P}_j^{(2)}(\theta_t\omega) u }_2. 
\end{equation*}
In fact, we have 
\begin{align*}
   1\le \norm{\widetilde{P}_j^{(2)}(\theta_t\omega)}_{2,2} & =   \sup \Big\{ \frac{\norm{\widetilde{P}_j^{(2)}(\theta_t\omega) u}_2 }{\norm{u}_2}: u\in X_2\setminus\{0\} \Big\} \\
    & \le  \sup \Big\{ \frac{\norm{\widetilde{P}_j^{(1)}(\theta_t\omega) i u}_1 }{c(t)\norm{iu}_1} \cdot \frac{\norm{iu}_1}{\norm{u}_2}: u\in X_2 \setminus\{0\} \Big\} \\
    & \le \frac{\norm{i}_{2,1}}{c(t)}\norm{\widetilde{P}_j^{(1)}(\theta_t\omega)}_{1,1}.  
\end{align*}
 Hence, $\ln \norm{\widetilde{P}^{(2)}(\theta_t\omega)}_{2,2}/t \to 0$ as $t\to\infty$. Furthermore, in the spirit of Tanny's theorem, when $t\to -\infty$, we can obtain $\ln \norm{\widetilde{P}^{(2)}(\theta_t\omega)}_{2,2}/t \to 0$ as well, see~\cite[Lemma C2]{GTQu}. In conclusion, let us observe that generally for the projection $\widetilde{P}_{1,\dots,l}^{(2)}(\theta_t\omega) $ onto the first component of~\eqref{eq:dec_X_2} we have   
 \begin{equation*}
    0 \le \frac{\ln \norm{\widetilde{P}_{1,\dots,l}^{(2)}(\theta_t\omega)}_{2,2}}{t}\le \frac{\ln l+   \max_{j} \ln \norm{ \widetilde{P}_{j}^{(2)}(\theta_t\omega)}_{2,2}}{t}. 
 \end{equation*}
 Hence, $\ln \norm{\widetilde{P}_{1,\dots,l}^{(2)}(\theta_t\omega)}_{2,2}/t\to 0$ as $t\to\infty$ and again by~\cite[Lemma C2]{GTQu} we can obtain this result for $t\to-\infty$.
\smallskip\par\noindent
\eqref{Osel:case_I}\eqref{O1:4} fix $j=1\dots k$, $\omega\in\Omega_0$ and $u\in U^{(1,2)}_{\theta_{-1}\omega}(1) E_j(\theta_{-1}\omega)$. Hence, from Claim~\ref{claim:E_2-subspace}\eqref{claim:sub-space-i} $iu\in E_j(\omega)$. Therefore, by assumption \eqref{Osel:case_I}\eqref{O1:4} for $\Phi^{(1)}$,  $\ln\norm{U^{(1)}_{\omega}(t)iu}_{1}/t\allowbreak\to\lambda_j$. Moreover, we have two inequalities, first one from the assumption~\eqref{a:commutation}: $\norm{U^{(1)}_{\omega}(t)\circ iu}_1 = \norm{i \circ U^{(2)}_{\omega}(t) u}_1 \le \norm{i}_{2,1} \norm{ U^{(2)}_{\omega}(t) u}_2$ 
and the second one, $\norm{U^{(2)}_{\omega}(t)u}_2 \le \norm{U^{(1,2)}_{\theta_{t-1}\omega}(1)}_{1,2} \, \norm{U^{(1)}_{\omega}(t-1)\circ i u }_1 $,  
from the fact that $U^{(2)}_{\omega}(t)u = U^{(1,2)}_{\theta_{t-1}\omega}(1)\circ U^{(1)}_{\omega}(t-1)\circ i u$. Therefore, 
\begin{multline*}
       \frac{1}{t}\ln\norm{U^{(1)}_{\omega}(t)\circ i u}_1 -\frac{1}{t}\ln\norm{i}_{2,1} \le \frac{1}{t}\ln\norm{U^{(2)}_{\omega}(t)u}_2 \\ \le \frac{1}{t}\ln\norm{U^{(1,2)}_{\theta_{t-1}\omega}(1)}_{1, 2} +\frac{1}{t}\ln\norm{U^{(1)}_{\omega}(t-1)\circ i u}_1.
\end{multline*}
Thus, we have proven \eqref{Osel:case_I}\eqref{O1:4} for $\Phi^{(2)}$. 
\smallskip\par\noindent
\eqref{Osel:case_I}\eqref{O1:7}
Fix $\omega$ and 
\begin{equation*}
    u^{(2)}\in  \bigoplus\limits_{m=j}^{k} U^{(1,2)}_{\theta_{-1}\omega}(1) E_m(\theta_{-1}\omega) \oplus i^{-1}(F_{\infty}(\omega)). 
\end{equation*}
To begin with, let us note that Claim~\ref{claim:E_2-subspace} allows for demonstrating only the equivalence between the existence of negative semiorbits (in $X_1$ and $X_2$) on which the appropriate Lyapunov exponent is achieved, since the membership of $u^{(2)}$ in the subspace $U^{(1,2)}_{\theta_{-1}\omega}(1) E_j(\theta_{-1}\omega)$ is equivalent to the membership of $iu^{(2)}$ in the subspace $ E_j(\omega)$. 

Assume that there exists a negative semiorbit $\tilde{u}^{(2)}(s)$ in $X_2$ passing through $(\omega, u^{(2)})$.  We commence by demonstrating that within the space $X_1$, there exists a negative semiorbit passing through $(\omega,iu^{(2)})$, such that the appropriate Lyapunov exponent is attainable. Namely, we define $\tilde{u}^{(1)}\colon (-\infty ,0] \to X_1$ by $\tilde{u}^{(1)}(s)=i\tilde{u}^{(2)}(s)$; such Lyapunov exponent in $X_2$ is (by assumption) $\lambda_j$. Upon examination, it becomes evident that the thus-defined $\tilde{u}^{(1)}$ indeed is a negative semi-orbit. Indeed $\tilde{u}^{(1)}(0)=iu^{(2)}$ and 
\begin{equation*}
  \tilde{u}^{(1)}(s+t) =  i \tilde{u}^{(2)}(s+t) = iU^{(2)}_{\theta_{s}\omega}(t) \tilde{u}^{(2)}(s) = U^{(1)}_{\theta_{s}\omega}(t) i \tilde{u}^{(2)}(s) = U^{(1)}_{\theta_{s}\omega}(t)   \tilde{u}^{(1)}(s). 
\end{equation*}
Moreover, $\tilde{u}^{(1)}$ is $(\mathfrak{B}((-\infty,0]),
\mathfrak{B}(X_1))$\nobreakdash-\hspace{0pt}measurable as a continuous composition. Furthermore, $\ln \norm{\tilde{u}^{(1)}(s)}_1/ s\to \lambda_j$ as $s\to-\infty$ since 
\begin{multline*}
    \norm{\tilde{u}^{(2)}(s+1)}_2 = \norm{U^{(1,2)}_{\theta_{s}\omega}(1)i \tilde{u}^{(2)}(s) }_2  \\
    \le \norm{U^{(1,2)}_{\theta_{s}\omega}(1)}_{1, 2} \,  \norm{i \tilde{u}^{(2)}(s) }_1 = \norm{U^{(1,2)}_{\theta_{s}\omega}(1)}_{1,2} \,  \norm{\tilde{u}^{(1)}(s)}_1 \\ 
    \le \norm{U^{(1,2)}_{\theta_{s}\omega}(1)}_{1,2} \norm{i}_{2,1} \norm{\tilde{u}^{(2)}(s)}_2. 
\end{multline*}
Conversely, if $\tilde{u}^{(1)}$ is a negative semiorbit on $X_1$, then $\tilde{u}^{(2)}$ defined as $\tilde{u}^{(2)}(s)=U^{(1,2)}_{\theta_{s-1}\omega}(1) \tilde{u}^{(1)}(s-1)$ serves as a negative semiorbit on $X_2$. Clearly,
\begin{equation*}
    i\tilde{u}^{(2)}(0) = iU^{(1,2)}_{\theta_{-1}\omega}(1) \tilde{u}^{(1)}(-1) = U^{(1)}_{\theta_{-1}\omega}(1) \tilde{u}^{(1)}(-1) = \tilde{u}^{(1)}(0) = u^{(1)}=iu^{(2)}, 
\end{equation*}
and for any  $s \le 0$, $t\geq 0$  such that  $s+t\le 0$ we have 
\begin{equation*}
    \begin{split}
    \tilde{u}^{(2)}(s+t) &= U^{(1,2)}_{\theta_{s-1+t}\omega}(1) \tilde{u}^{(1)}(s-1+t) \\
    &=U^{(1,2)}_{\theta_{s-1+t}\omega}(1) \circ U^{(1)}_{\theta_{s-1}\omega}(t)  \tilde{u}^{(1)}(s-1) \\
    &=  U^{(2)}_{\theta_{s}\omega}(t)\circ  U^{(1,2)}_{\theta_{s-1}\omega}(1)  \tilde{u}^{(1)}(s-1) \\
    &=  U^{(2)}_{\theta_{s}\omega}(t)    \tilde{u}^{(2)}(s).
\end{split}
\end{equation*}
Moreover, $\ln \norm{\tilde{u}^{(2)}(s)}_2/ s\to \lambda_j$ as $s\to-\infty$ since
\begin{equation*}
    \norm{\tilde{u}^{(1)}(s)}_1 \le \norm{i}_{2,1} \, \norm{\tilde{u}^{(2)}(s)}_2 \le \norm{i}_{2,1} \, \norm{U^{(1,2)}_{\theta_{s-1}\omega}(1)}_{1,2} \, \norm{\tilde{u}^{(1)}(s-1)}_1. 
\end{equation*}
It remains to show the $(\mathfrak{B}((-\infty,0]),
\mathfrak{B}(X_2))$\nobreakdash-\hspace{0pt}measurability of the map  
\begin{equation*}
    \big[ \, (-\infty,0]  \ni s \mapsto  U^{(1,2)}_{\theta_{s-1}\omega}(1) \tilde{u}^{(1)}(s-1) \in X_2\, \big].
\end{equation*}
For this purpose, let us observe that it can be rewritten as the measurable composition 
\begin{equation*}
    s\mapsto (\theta_{s-1}\omega,s)\mapsto (U_{\theta_{-1}\omega}^{(1,2)}(1),  \tilde{u}^{(1)}(s-1) ) \mapsto U^{(1,2)}_{\theta_{s-1}\omega}(1) \tilde{u}^{(1)}(s-1). 
\end{equation*}
The first mapping is $(\mathfrak{B}((-\infty,0]),
\mathfrak{F}\otimes \mathfrak{B}((-\infty,0]))$\nobreakdash-\hspace{0pt}measurable, the second mapping is $(\mathfrak{F}\otimes \mathfrak{B}((-\infty,0]), \mathfrak{B}(\mathcal{L}_{\mathrm{s}}(X_1,X_2))\otimes\mathfrak{B}(X_1))$\nobreakdash-\hspace{0pt}measurable. By the assumption~\eqref{a:separable} and Lemma~\ref{lemma:strong_measurable_paring}, the pairing is $(\mathfrak{B}(\mathcal{L}_{\mathrm{s}}(X_1,X_2))\otimes\mathfrak{B}(X_1), \allowbreak\mathfrak{B}(X_2) )$\nobreakdash-\hspace{0pt}measurable. 

\smallskip\par\noindent
\eqref{Osel:case_I}\eqref{O1:8}
Firstly, we claim that
\begin{gather*}
U^{(2)}_{\omega}(t)= U^{(1,2)}_{\theta_{t-1}\omega}(1)\circ U^{(1)}_{\omega}(t-1) \circ i. 
\end{gather*}
Together with $i(i^{-1}(F_{\infty}(\omega)))\subset F_{\infty}(\omega)$ we can see that the inequality 
\begin{align*}
    \norm{U^{(2)}_{\omega}(t){\restriction}_{i^{-1}(F_{\infty}(\omega))}}_{2,2} & \le  \norm{i}_{2,1} \, \norm{U^{(1,2)}_{\theta_{t-1}\omega}(1)}_{1,2} \,  \norm{U^{(1)}_{\omega}(t-1){\restriction}_{i(i^{-1}(F_{\infty}(\omega)))}}_{1,1} \\
    &\le \norm{i}_{2,1} \, \norm{U^{(1,2)}_{\theta_{t-1}\omega}(1)}_{1,2} \, \norm{U^{(1)}_{\omega}(t-1){\restriction}_{F_{\infty}(\omega)}}_{1,1}
\end{align*}
holds for any $\omega\in\Omega$ and $t\ge 1$. So, 
\begin{multline*}
         \frac{1}{t}\ln \norm{U^{(2)}_{\omega}(t){\restriction}_{i^{-1}(F_{\infty}(\omega))}}_{2,2} \le \frac{1}{t}\ln\norm{i}_{2,1} + \frac{1}{t}\ln\norm{U^{(1,2)}_{\theta_{t-1}\omega}(1)}_{1,2} + \frac{1}{t}\ln \norm{U^{(1)}_{\omega}(t-1){\restriction}_{F_{\infty}(\omega)}}_{1,1}, 
\end{multline*}
with under assumptions~\eqref{a:emdeding} and \eqref{a:connector} suffices to finish the proof of this part. 

Furthermore, we have concluded the first and, as it turns out, the main part of the proof. The remaining part \eqref{Osel:case_II} pertains to the scenario in which there is a countable number of Lyapunov exponents. Nevertheless, it is easy to notice that practically the entire proof is replicated word for word in this situation. 
However, for the sake of completeness, let us note the fact that if $\Phi^{(1)}$ admits Oseledets decomposition in the case \eqref{Osel:case_II} with parameters 
\begin{equation*}
     (\Omega_0, (\lambda_j)_{j\in\NN}, \{E_j(\omega)\}_{\omega\in\Omega_0,j\in\NN}, \{ F_{j}(\omega)\}_{\omega \in \Omega_0, j\in\NN}) 
 \end{equation*}
then $\Phi^{(2)}$ also admits Oseledets decomposition in the case \eqref{Osel:case_II}, with parameters 
\begin{equation*}
     (\Omega_0, (\lambda_j)_{j\in\NN}, \{U^{(1,2)}_{\theta_{-1}\omega}(1) E_j(\theta_{-1}\omega)\}_{\omega\in\Omega_0,j\in\NN}, \{i^{-1}(F_{j}(\omega))\}_{\omega \in \Omega_0, j\in\NN}). 
 \end{equation*} 
\end{proof}

\section*{Assumption discussion}
\label{sec:Appendix_B}
In this section, we aim to discuss the scenario where assumption~\eqref{a:sigma-compatibility} is met. Indeed, we will show something more, namely that this assumption is fulfilled when, in particular, $X_1=L_p([-1,0])$ ($1<p<\infty$) and $X_2=C([-1,0])$, and that measurability of function factors follows from it. 
\begin{lemma}
 Let $i:C([-1,0])\to L_p([-1,0])$ be the continuous injection defined by $i(f)=f$. Moreover, let $w^{(1)}:\Omega\to L_p([-1,0])$ be a $(\mathfrak{F},\mathfrak{B}(L_p([-1,0])))$-measurable function with the factor $w^{(2)}:\Omega\to C([-1,0])$, i.e. $w^{(1)} =i \circ w^{(2)} $. Then $w^{(2)}:\Omega\to C([-1,0])$ is $(\mathfrak{F},\mathfrak{B}(C([-1,0])))$-measurable. 
\end{lemma}

\begin{proof}
The proof is based on observation that every closed ball, specifically the unit ball, in $C([-1,0])$ corresponds to a pre-image of $i$ of some $L_p$-closed subset $F\subset L_p([-1,0])$, i.e. $\bar{B}_{C([-1,0])}(0,1)=i^{-1}(F)$ for some $L_p$-closed $F$. The set $F\vcentcolon=\{i(f)\in L_p([-1,0]):-1\le f(x)\le 1 \text{ for a.e. } x\in [-1,0]\}$ proves to be a witness, as $L_p([-1,0]) \setminus F$ is $L_p$-open. Indeed, for $g\in L_p([-1,0]) \setminus F$ there exists a Lebesgue-measurable set $M\subset [-1,0]$ such that its Lebesgue measure $\lambda(M)>0$ and $\epsilon>0$ such that $g(x)\ge 1+\epsilon$ for all $x\in M$. Therefore, for $f\in F$ we have $\|g-f\|_{L_p}\ge \epsilon \lambda^{1/p}(M)$ so $B_{L_p(-1,0)}(g,\epsilon \lambda^{1/p}(M))\subset L_p([-1,0]) \setminus F$. It remains to show $i^{-1}(F)=\bar{B}_{C([-1,0])}(0,1)$. However, it can be easily done via definition of pre-image since, conditions $-1\le f(x)\le 1$ for a.e. $x$ and $-1\le f(x)\le 1$ for all $x$ are equivalent for continuous functions. 
\smallskip\par
As a final step, we demonstrate that the map $w^{(2)}$ is $(\mathfrak{F},\mathfrak{B}(C([-1,0])))$-measurable. Let $D$ be an open subset of $C([-1,0])$ in the sense of the $C([-1,0])$ topology. Since $C([-1,0])$ is separable, it has a countable base, and $D$ can be expressed as a countable union of open balls. Moreover, open balls can be written as countable union of closed ones. As a consequence the set $D$ is a countable union of pre-images of $L_p$-closed sets. The observation that 
\begin{equation*}
    (w^{(2)})^{-1}[D]=(w^{(2)})^{-1}\Big( \bigcup_{n\in\NN} i^{-1}(F_n) \Big)=(i\circ w^{(2)})^{-1}\Big(\bigcup_{n\in\NN} F_n\Big)=(w^{(1)})^{-1}\Big (\bigcup_{n\in\NN} F_n\Big )\in\mathfrak{F} 
\end{equation*}
finishes the proof. 
\end{proof}

\textbf{Acknowledgment.}  

I am grateful to Janusz \JM\ for inspiring conversations.

\end{document}